\documentclass[10pt,fleqn]{article}
\usepackage{mathptmx}

\input{epsf.tex}
\usepackage{latexsym,amsfonts,amssymb,epsfig,verbatim}
\usepackage{amsmath,amsthm,amssymb,latexsym,graphics,textcomp}
\usepackage{eucal,eufrak}
\usepackage{graphicx,color}
\usepackage{url}
\usepackage[T1]{fontenc} 

\input xy
\xyoption{all}

\theoremstyle{plain}
\newtheorem{theorem}{Theorem}
\newtheorem{lemma}[theorem]{Lemma}
\newtheorem{corollary}[theorem]{Corollary}
\newtheorem*{thm}{Theorem}

\theoremstyle{definition}

\newtheorem*{question}{Question}

\newcommand{\co}{\colon\thinspace}

\newcommand{\G}{\mathcal G}
\newcommand{\g}{\mathcal G_{H\cap K}}
\newcommand{\T}{\mathcal T}

\newcommand{\GH}{\Gamma_{\! H}}
\newcommand{\GK}{\Gamma_{\! K}}
\newcommand{\GHK}{\Gamma_{\! H\cap K}}
\newcommand{\GHvK}{\Gamma_{\! H\vee K}}

\newcommand{\X}{\mathcal X}
\newcommand{\rank}{\mathrm{rank}}
\newcommand{\val}{\mathrm{valence}}
\newcommand{\st}{\mathrm{star}}
\newcommand{\diagmatrix}{\vspace{4pt}\! \! \! \!  {{}_{\mbox{$1$}}}_{\, \ddots_{\mbox{\! \! $1$}}}}

\date{August 14, 2008}

\begin{document}

\title{\textbf{Intersections and joins of free groups}}
\author{Richard P. Kent IV\thanks{Work supported by a Donald D. Harrington Dissertation Fellowship and a National Science Foundation Postdoctoral Fellowship.}}
\maketitle

\begin{flushright}
{\small
\textit{The possible ranks higher than the actual.}

---common paraphrase of M. Heidegger.
}
\end{flushright}


\section{Introduction}

\noindent Let $F$ be a free group.  If $H$ and $K$ are subgroups of $F$, we let $H\vee K = \langle H, K \rangle$ denote the \textbf{join} of $H$ and $K$.

We study the relationship between the rank of $H \cap K$ and that of $H\vee K$ for a pair of finitely generated subgroups $H$ and $K$ of $F$.
In particular, we have the following particular case of the Hanna Neumann Conjecture, which has also been obtained by L. Louder \cite{louder} using his machinery for folding graphs of spaces \cite{louderKrull1,louderKrull2,louderfolding}.
For detailed discussions of the Hanna Neumann Conjecture, see \cite{hanna,hannaaddendum,walter,stallings,gersten,dicks}.
\begin{theorem}[Kent, Louder]\label{particulartheorem} Let $H$ and $K$ be nontrivial finitely generated subgroups of a free group of ranks $h$ and $k$, respectively.
If 
\[
\rank(H\vee K) - 1 \geq \frac{h+k-1}{2}
\]
then
\[
\rank(H\cap K) - 1 \leq (h-1)(k-1).
\]
\end{theorem}

We also give a new proof of R. Burns' theorem \cite{burns}:
\begin{thm}[Burns] Let $H$ and $K$ be nontrivial finitely generated subgroups of a free group with ranks $h$ and $k$, respectively.
Then
\[
\rank(H\cap K) - 1 \leq 2(h - 1)(k - 1) - \min\big\{(h - 1), (k - 1)\big\}.
\]
\end{thm}
\noindent (In fact, we obtain W. Neumann's form of this inequality \cite{walter}, see Section \ref{burnssection}.)

Our main theorem is the following strong form of Burns' inequality:
\begin{theorem}\label{strongburns} Let $H$ and $K$ be nontrivial finitely generated subgroups of $F$ of ranks $h$ and $k \geq h$, respectively, that intersect nontrivially.  Then
\[
\rank(H\cap K) -1 \ \leq \ 2(h -1)(k-1) - (h -1)\big(\rank(H\vee K) - 1\big).
\]
\end{theorem}
\noindent This theorem, with an additional hypothesis, is claimed by W. Imrich and T. M\"uller in \cite{imrichmuller}.  
Unfortunately, their proof contains an error---see the end of the Section \ref{backgroundsection} for a detailed discussion.  
Note that the hypothesis on the intersection cannot be dispensed with entirely, for when $h=k \geq 3$, the inequality will fail if $\rank(H\vee K) = 2k$---but this is the only situation in which it fails.

\bigskip
\noindent
We were brought to Theorem \ref{strongburns} by the following question of M. Culler and P. Shalen.
\begin{question} If $H$ and $K$ are two rank--$2$ subgroups of a free group and $H \cap K$ has rank two, must their join have rank two as well?
\end{question}

\noindent An affirmative answer follows immediately from Theorem \ref{strongburns}, and we record this special case as a theorem---this has also been derived using Louder's folding machine by Louder and D. B. McReynolds \cite{louder}, independently of the work here.\footnote{This theorem was proven by both parties before Theorems \ref{particulartheorem} and \ref{strongburns} were proven. }
\begin{theorem}[Kent, Louder--McReynolds]\label{main} Let $H$ and $K$ be rank--2 subgroups of a free group $F$.  
Then 
\[
\rank(H\cap K) \leq 4 - \rank(H\vee K).
\]
\end{theorem}
\noindent
In \cite{loudermcreynolds}, Louder and McReynolds also give a new proof of W. Dicks' theorem \cite{dicks} that W. Neumann's strong form of the Hanna Neumann Conjecture is equivalent to Dicks' Amalgamated Graph Conjecture.

Theorem \ref{main} allows Culler and Shalen to prove the following, see \cite{cullershalen}.
Recall that a group is \textbf{$k$--free} if all of its $k$--generator subgroups are free.

\begin{thm}[Culler--Shalen] Let $G$ be a $4$--free Kleinian group.
Then there is a point $p$ in $\mathbb H^3$ and a cyclic subgroup $C$ of $G$ such that for any element $g$ of $G-C$, the distance between $p$ and $g p$ is at least $\log 7$.
\end{thm}

\noindent This has the following consequence, see \cite{cullershalen}.
\begin{thm}[Culler--Shalen] Let
$M$ be a closed orientable hyperbolic $3$--manifold such that $\pi_1(M)$ is $4$--free.  Then the volume of $M$ is at least $3.44$.
\end{thm}

Theorem  \ref{main} is sharp in that, given nonnegative integers $m$ and $n$ with $n\geq 2$ and $m \leq 4-n$, there are $H$ and $K$ of rank two with $\rank(H\cap K) = m$ and $\rank(H\vee K) = n$.

To see this, note that, by Burns' theorem, the rank of $H \cap K$ is at most two.

If $H \vee K$ has rank four, then, since finitely generated free groups are Hopfian, we have $H \vee K = H * K$, and hence $H \cap K = 1$.  

If the join has rank two, $H \cap K$ may have rank zero, one, or two.  For completeness, we list examples. 
If $H=K$, then $H\cap K = H = H\vee K$.  
If $H=\langle a, bab \rangle$ and $K=\langle b,a^2 \rangle$, the join is $\langle a,b \rangle$ and the intersection is $\langle a^2 \rangle$. 
If $H=\langle a, bab \rangle$ and $K = \langle b, ab^{-1}aba^{-1} \rangle$, then $H \cap K =1$ and the join is $\langle a, b \rangle$.

Finally, there are rank two $H$ and $K$ whose join has rank three and whose intersection is trivial.  For example, consider the free group on $\big\{a,b,c\}$ and let $H = \langle c, a^{-1} b a\rangle$ and $K = \langle a, b^{-1}cb \rangle$.  Of course, there are rank two $H$ and $K$ whose intersection is infinite cyclic and whose join has rank three, like $\langle a, b \rangle$ and $\langle b, c \rangle$ in a free group on $\{a,b,c\}$.

\subsection*{Perspective}
 
The heart of the work here lies in the study of a certain pushout and the restraints it places on the rank of the intersection $H \cap K$. 
The pictures that emerge here and in the work of Louder and McReynolds \cite{loudermcreynolds} share a common spirit, and both are akin to the work of W. Dicks \cite{dicks}.
The arguments here are chiefly combinatorial; those of \cite{loudermcreynolds} more purely topological. 
Whilst having the same theoretical kernel, the two discussions each have their own merits, and the authors have decided to preserve them in separate papers. 

\bigskip
\noindent \textbf{Acknowledgments.} The author thanks Warren Dicks, Cameron Gordon, Wilfried Imrich, Lars Louder, Joe Masters, Ben McReynolds, Walter Neumann, and Alan Reid for lending careful ears.  He thanks Ben Klaff for bringing Culler and Shalen's question to his attention.
The author also extends his thanks to the referee for many thoughtful comments that have improved the exposition tremendously.

When the author first established Theorem \ref{main}, he used the pushout of $\GH$ and $\GK$ along the component of $\g$ carrying the group $H \cap K$, rather than the pushout $\T$ along the core $\GHK$---the former is somewhat disagreeable, and may possess special vertices.  
In correspondence with Louder, it was Louder's use of the core $\GHK$ that prompted the author's adoption of the graph $\T$. 
The author thus extends special thanks to Louder.

\section{Graphs, pullbacks, and pushouts}\label{backgroundsection}
We may assume that $F$ is free on the set $\{a, b\}$, and we do so. 
We identify $F$ with the fundamental group of a wedge $\X$ of two circles based at the wedge point, and we orient the two edges of $\X$.  

We have distilled here the notions of \cite{stallings} and \cite{gersten} into a form that is convenient for our purpose.

Given a subgroup $H$ of $F$, there is a covering space $\widetilde \X_H$ corresponding to $H$. 
There is a unique choice of basepoint $*$ in $\widetilde \X_H$ so that $\pi_1(\widetilde \X_H, *)$ is identical to $H$. 
We let $\GH$ denote the smallest subgraph of $\widetilde \X_H$ containing $*$ that carries $H$. 
The graph $\GH$ comes naturally equipped with an \textbf{oriented labeling}, meaning that each edge is oriented and labeled with an element of $\{a, b\}$. 
The orientation of a given edge $e$ yields an \textbf{initial vertex} $\iota(e)$ and a \textbf{terminal vertex} $\tau(e)$, which may or may not be distinct.

The graphs so far discussed are labeled \textbf{properly}, meaning that if edges $e$ and $f$ have the same labeling and either $\iota(e) = \iota(f)$ or $\tau(e) = \tau(f)$, then the two edges agree.

The \textbf{star} of a vertex $v$, written $\st(v)$, is the union of the edges incident to $v$ equipped with the induced oriented labeling. 
The \textbf{valence} of a vertex $v$ is the number of edges incident to $v$ counted with multiplicities. 
All of the above graphs are at most $4$--valent, meaning that their vertices have valence at most four. 
A vertex is a \textbf{branch vertex} if its valence is at least $3$. 
We say that a vertex is \textbf{extremal} if its valence is less than or equal to one.  
We say that a graph is \textbf{$k$--regular} if all of its \textit{branch} vertices have valence $k$.

A \textbf{map of graphs} between two oriented graphs is a map that takes vertices to vertices, edges to edges, and preserves orientations. 
A map of graphs is an \textbf{immersion} if it is injective at the level of edges on all stars.
A \textbf{labeled map of graphs} between two labeled oriented graphs is a map of graphs that preserves labels.
A  \textbf{labeled immersion} is a labeled map of graphs that is also an immersion.

Two $k$--valent vertices of labeled oriented graphs are of the \textbf{same type} if there is a labeled immersion from the star of one to the star of the other.

J. Stallings' category of oriented graphs is the category whose objects are oriented graphs (without labelings), and whose morphisms are maps of graphs---S. Gersten's category has the same objects, but more maps \cite{gersten}.
The collection of all oriented graphs with labels in $\{a, b\}$ together with all labeled maps of graphs form a category that we call the \textbf{category of labeled oriented graphs}---there is an obvious forgetful functor into Stallings' category.

We will also consider the \textbf{category of properly labeled oriented graphs}, whose objects are properly labeled oriented graphs and whose morphisms are labeled immersions.

Given a graph $\Gamma$, let $V(\Gamma)$ be its set of vertices.

We define a graph $\g$ as follows.  
Its set of vertices is the product $V(\GH) \times V(\GK)$ and there is an edge labeled $x$ joining $(a,b)$ to $(c,d)$ oriented from $(a,b)$ to $(c,d)$ if and only if there is an edge in $\GH$ labeled $x$ joining $a$ to $c$ oriented from $a$ to $c$ \textit{and} an edge in $\GK$ labeled $x$ joining $b$ to $d$ oriented from $b$ to $d$. 
The graph $\g$ is the \textbf{fiber product} of the maps $\GH \to \X$ and $\GK \to \X$---in other words, the pullback of the diagram
\[
 \xymatrix{  &   \GH \ar[d]    \\
  \GK \ar[r] & \X
    } 
\]
 in the category of oriented graphs---it is also the pullback in the category of properly labeled oriented graphs, and in this category, it is in fact the direct product $\GH \times \GK$.
 
 The graph $\GHK$ is a subgraph of $\g$, and carries the fundamental group \cite{stallings}.

Note that there are projections $\Pi_H\co \g \to \GH$ and $\Pi_K\co \g \to \GK$ and that a path $\gamma$ from $(*,*)$ to $(u,v)$ in $\g$ projects to paths $\Pi_H(\gamma)$ and $\Pi_K(\gamma)$ with the same labeling from $*$ to $u$ and $*$ to $v$, respectively. Conversely, given two pointed paths $\gamma_H$ and $\gamma_K$ with identical oriented labelings from $*$ to $u$ and $*$ to $v$ respectively, there is an identically labeled path $\gamma$ in $\g$ from $(*,*)$ to $(u,v)$ that projects to $\gamma_H$ and $\gamma_K$.

Given a graph $\Gamma$ with an oriented (nonproper) labeling, a \textbf{fold} is the following operation: if $e_1$ and $e_2$ are two edges of $\Gamma$ with the same label and $\iota(e_1) = \iota(e_2)$ or $\tau(e_1)=\tau(e_2)$, identify $e_1$ and $e_2$ to obtain a new graph.  

The properly labeled graph $\GHvK$ is obtained from $\GH$ and $\GK$ by forming the wedge product of $\GH$ and $\GK$ at their basepoints and folding until no more folding is possible. 

 In what follows, we identify $\GH$ and $\GK$ with their images in $\GH \sqcup \GK$ whenever convenient.

 The graph $\GHvK$ is the pushout in the category of \textit{properly labeled} oriented graphs of the diagram
 \[
 \xymatrix{ \ast   \ar[r]  \ar[d]
 &   \GH   \\
  \GK
    } 
\]
where the single point $\ast$ maps to the basepoints in $\GH$ and $\GK$.
This category is somewhat odd in that $\GHvK$ is also the pushout of 
\[
 \xymatrix{\GHK   \ar[r]  \ar[d]
 &   \GH   \\
  \GK
    } 
\]

We will make use of a labeled oriented graph that is not properly labeled.  This is the \textbf{topological pushout} $\T$ of the diagram
\[
 \xymatrix{\GHK   \ar[r]  \ar[d]
 &   \GH   \\
  \GK
    } 
\]
The letters $x$ and $y$ will denote points in $\GH$ and $\GK$, respectively.
The graph $\T$ is the quotient of $\GH \sqcup \GK$ by the equivalence relation $\mathfrak{R}$ \textit{generated} by the relations $x \sim y$ if $x \in \Pi_H\big({\Pi_K}^{\! \! \! -1}(y)\big)$ or $y \in \Pi_K\big({\Pi_H}^{\! \! \! -1}(x)\big)$. 

So, points $a,b \in \GH \sqcup \GK$  map to the same point in $\T$ if and only if there is a sequence $\{(x_i,y_i)\}_{i=1}^n$ in $\GHK$ such that $a$ is a coordinate of $(x_1,y_1)$, $b$ is a coordinate of $(x_n,y_n)$, and for each $i$ either $x_i = x_{i+1}$ or $y_i = y_{i+1}$. 
We call such a sequence a \textbf{sequence for $a$ and $b$}.
Note that a minimal sequence for $x$ and $y$ will not have $x_i = x_j=x_k$ or $y_i = y_j=y_k$ for any pairwise distinct $i$, $j$, and $k$.

We warn the reader that the equivalence relation on $\GH \sqcup \GK$ whose quotient is $\GHvK$ is typically coarser than the one just described.
For instance, in the example in Figure \ref{counterexample}, $\GHvK$ is $\X$, but $\T$ is not.
\begin{figure}
\begin{minipage}{1\textwidth}
\ \ 
\includegraphics{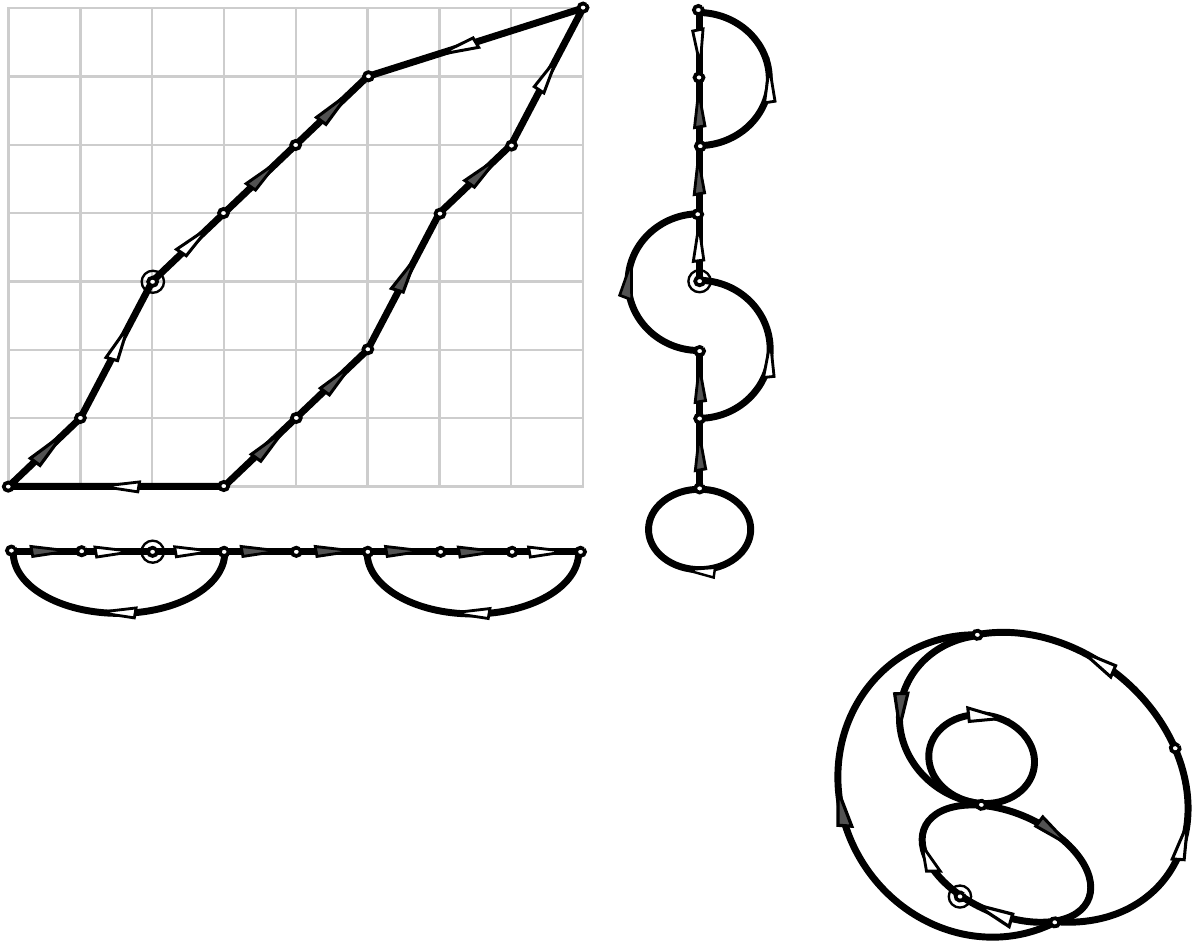}
\end{minipage}
\begin{minipage}{.665\textwidth}
\vspace{-88pt}
\caption{Above, 
the graph
$\GK$ is at the bottom, $\GH$ at the right.
Their basepoints are encircled.
White arrows correspond to $a$, black arrows $b$.
Writing $x^g = gxg^{-1}$,
we have 
$K = \langle a^2ba,   (b^2a^2)^{ab^2} \rangle$, 
$H = \langle  ab^{-2}a,  (ba^{-2})^{ab}, a^{a^{-1}b^{-1}}\rangle$,
${H \cap K} = \langle ab^2a^{-2}b^{-4}aba\rangle$, 
and $H\vee K = \langle a, b \rangle$. 
The graph $\T$ is to the right.
Notice that $\chi(\T) = -3 < -1 = \chi(\GHvK)$.
}\label{counterexample}
\end{minipage}
\end{figure}

The graph $\T$ is also the pushout in the category of labeled oriented graphs, but not necessarily the pushout in the category of \textit{properly} labeled oriented graphs---again, see Figure \ref{counterexample}.

Though not equal to $\GHvK$ in general, 
$\T$ does fit into the commutative diagram
\[
 \xymatrix{\GHK  \ar[r]^{\Pi_H}  \ar[d]_{\Pi_K} &   \GH   \ar[d] \ar[ddr] \\
  \GK \ar[r] \ar[drr] & \T \ar[dr] \\
  & & \GHvK
    } 
\]
where the map $\T \to \GHvK$ factors into a series of folds.
As a fold is surjective at the level of fundamental groups, see \cite{stallings}, we have $\chi(\T) \leq \chi(\GHvK)$.

\bigskip
\noindent
Confusing $\T$ and $\GHvK$ can be hazardous, and we call $\T$ the topological pushout to prevent such confusion. 
This is the source of the error in \cite{imrichmuller}, which we now discuss.

The proof of the lemma on page 195 of \cite{imrichmuller} is incorrect.
The error lies in the last complete sentence of that page:
\begin{quote}
In order that both $x$ and $y$ be mapped onto $z$ there must be a sequence
\[
x = x_0, x_1, x_2, \ldots, x_n = y
\]
of vertices of $\Gamma_0$ such that for every $i$ $x_i$ and $x_{i+1}$ have the same image in either $\Gamma_1$ or $\Gamma_2$ (and all are mapped to $z$ in $\Delta$).
\end{quote}
Here $\Gamma_0$ is our graph $\GHK$, the graphs $\Gamma_1$ and $\Gamma_2$ are our graphs $\GH$ and $\GK$, and the graph $\Delta$ is our $\GHvK$. 
Here is a translation of this into our terminology:
\begin{quote}
Let $z$ be a vertex in $\GHvK$ and let $a$ and $b$ be vertices of $\GH \sqcup \GK$ that map to $z$. In order that both $a$ and $b$ be mapped onto $z$, there must be a sequence for $a$ and $b$.
\end{quote}
This statement is false. 
The example in Figure \ref{counterexample} is a counterexample: the graph
$\GHvK$ is the wedge of two circles with a vertex $z$, say, and so all vertices in $\GH \sqcup \GK$ map to $z$ under the quotient map $\GH \sqcup \GK \to \GHvK$; on the other hand, the basepoints for $\GH$ and $\GK$ are the only vertices in their $\mathfrak{R}$--equivalence class---as is easily verified by sight.

The statement \textit{is} correct once $\GHvK$ has been replaced by $\T$, but, unfortunately, the arguments in \cite{imrichmuller} rely on the fact that $\GHvK$ is $3$--regular, a property that $\T$ does not generally possess.

The lemma in \cite{imrichmuller} would be quite useful, and though its proof is incorrect, we do not know if the lemma actually fails.\footnote{\textit{History of the error:} In the Fall of 2005, the author produced a faulty proof of Theorem \ref{main}.  Following this, he discovered the paper \cite{imrichmuller}, from which Theorem \ref{main} would follow. Unable to prove the existence of the sequence $x = x_0, x_1, x_2, \ldots, x_n = y$ in the quoted passage, the author contacted Imrich.  Amidst the resulting correspondence, the author found the example in Figure \ref{counterexample}.}

\section{Estimating the Euler characteristic of $\T$}

Let $H$ and $K$ be subgroups of $F$ of ranks $h$ and $k$. 
Suppose that $H \cap K \neq 1$.

For simplicity, we reembed $H\vee K$ into $F$ so that all branch vertices in $\GHvK$ are $3$--valent and of the same type: we replace $H\vee K$ with its image under the endomorphism $\varphi$ of $F$ defined by $\varphi(a) = a^2$, and $\varphi(b) = [a,b] = aba^{-1}b^{-1}$. 
Note that this implies that all branch vertices of $\GH$ and $\GK$ are $3$--valent and of the same type.

If a restriction of a covering map of graphs fails to be injective on an edge, then the edge must descend to a cycle of length one.
So our normalization above guarantees that the restriction of the quotient $\GH \sqcup \GK \to \GHvK$ to any edge is an embedding (as the target has no unit cycles), and hence \textit{the restriction of the quotient $\GH \sqcup \GK \to \T$ to any edge is an embedding}.

We claim that it suffices to consider the case where neither $\GH$ nor $\GK$ possess extremal vertices.  
It is easy to see that by conjugating $H\vee K$ in $F$, one may assume that $\GH$ has no extremal vertices, and we assume that this is the case.
Let $p$ and $q$ be the basepoints of $\GH$ and $\GK$, respectively.
Suppose that $q$ is extremal.  
Let $\gamma$ be the shortest path in $\GK$ starting at $q$ and ending at a branch vertex. 
Suppose that $\gamma$ is labeled with a word $w$ in $F$.
Since $H \cap K$ is not trivial, the graph $\GHK$ contains a nontrivial loop based at $(p,q)$, and so there is a path $\delta$ in $\g$ starting at $(p,q)$ labeled $w$.  
Now $\delta$ projects to a path in $\GH$ starting at $p$ that is labeled $w$.  
This means that if we conjugate $H\vee K$ by $v = w^{-1}$, the graphs $\Gamma_{\! H^v}$ and $\Gamma_{\! K^v}$ have no extremal vertices, and of course, $\rank\big(H^v \cap K^v\big) = \rank\big((H \cap K)^v\big) = \rank(H\cap K)$, and $\rank\big((H\vee K)^v\big) = \rank(H\vee K)$.

We assume these normalizations throughout.
Note that since $\GH$ and $\GK$ have no extremal vertices, neither does $\GHK$.


\subsection{Stars}

If $\Gamma$ is a graph, let $\mathbf{b}(\Gamma)$ denote the number of branch vertices in $\Gamma$.
If $\Gamma $ is $3$--regular, then $-\chi(\Gamma) = \rank\big(\pi_1(\Gamma)\big) -1=\mathbf{b}(\Gamma)/2$.

Consider the topological pushout $\T$ of $\GH$ and $\GK$ along $\GHK$, and the equivalence relation $\mathfrak{R}$ on $\GH \sqcup \GK$ that defines it. 
Again, $-\chi(\T) \geq \rank(H\vee K) -1$.

This section is devoted to the proof of the following theorem---compare Lemma 5.3 of \cite{louderfolding}.
We estimate the Euler characteristic of $\T$ by studying the set of $\mathfrak{R}$--equivalence classes of stars.
The equivalence class of the star of a vertex $b$ in $\GH \sqcup \GK$ is denoted by $[\st(b)]_\mathfrak{R}$.
If $X$ is a set, $\#X$ will denote its cardinality.
\begin{theorem}\label{eulertheorem}
\begin{equation}\label{euler}
-\chi(\T) \leq \frac{1}{2} \# \big\{ [\st(b)]_\mathfrak{R}\ \big | \ b \in \GH \sqcup \GK \  \mathrm{and}\ \mathrm{valence}(b)=3 \big\}.
\end{equation}
\end{theorem}

In the following, we will denote the type of a $2$--valent vertex in $\GH$, $\GK$, or $\GHK$ by a Roman capital.

We say that a vertex $z$ is \textbf{special} if it is a branch vertex of $\T$ that is not the image of a branch vertex in $\GH$ or $\GK$---we will show that there are no such vertices.

\begin{lemma}\label{notallsame} Let $z$ be a special vertex of $\T$.  Then there are vertices $a$ and $b$ in $\GH \sqcup \GK$ that have different types and get carried to $z$. 
\end{lemma}
\begin{proof} Suppose to the contrary that any $a$ and $b$ that get carried to $z$ have the same type.  

Let $a$ and $b$ be such a pair and let $\{(x_i,y_i)\}$ be a sequence for $a$ and $b$.  
Since $z$ is special, all of the $x_i $ and $y_i $ are $2$--valent.  By our assumption, all of the $x_ i $ and $y_ i $ have the same type. 
But this means that the $(x_i,y_i)$ are all $2$--valent and of the same type.
This means, in turn, that the stars of all the $x_i $ and $y_i $ get identified in $\T$.

This contradicts the fact that $z$ was a branch vertex.
\end{proof}

\begin{corollary}\label{nospecials} There are no special vertices in $\T$.
\end{corollary}
\begin{proof} Let $z$ be a special vertex. 

By Lemma \ref{notallsame}, there are vertices $a$ and $b$ of types $A$ and $B \neq A$ that map to $z$.

Let $\{v_i\}_{i=1}^n$ be a sequence for $a$ and $b$.  
The vertex $v_1$ has type $A$, and $v_n$ has type $B$.
Somewhere in between, the types must switch, and by the definition of sequence, we find a $v_j$ with a coordinate of type $A$, and a coordinate of type $X \neq A$.
This implies that $v_j$ is extremal.
But $\GHK$ has no extremal vertices.
\end{proof}


\begin{lemma}\label{staridentification} Let $z$ be a branch vertex in $\T$. Let $\G^z$ be the subgraph of $\T$ obtained by taking the union of the images of the stars of all branch vertices in $\GH \sqcup \GK$ mapping to $z$. 
If $\val_{\G^z}(z)$ is the valence of $z$ in $\G^z$, then
\[
\val_{\G^z}(z) \leq 2 + \#\big\{ [\st(b)]_\mathfrak{R}\ \big | \ b \in \GH \sqcup \GK  \mathrm{,}\ \mathrm{valence}(b)=3   \mathrm{,\ and}\ b \mapsto z
\big\}.
\]
\end{lemma}
\begin{proof} Let 
\[
n = \#\big\{ [\st(b)]_\mathfrak{R}\ \big | \ b \in \GH \sqcup \GK   \mathrm{,}\ \mathrm{valence}(b)=3   \mathrm{,\ and}\ b \mapsto z
\big\}
\]
and let $b_1, \ldots, b_n$ be a set of branch vertices whose stars form a set of representatives for the set
$
\big\{ [\st(b)]_\mathfrak{R}\ \big | \ b \in \GH \sqcup \GK  \mathrm{,}\ \mathrm{valence}(b)=3   \mathrm{,\ and}\ b \mapsto z
\big\}.
$

For $1\leq j \leq n$, let $\G_j$ be the union of the images in $\T$ of the stars of $b_1, \ldots, b_j$. So, $\G_n = \G^z$.

Since the restriction of $\GH \sqcup \GK \to \T$ to any edge is an embedding, the valence of $z$ in $\G_1$ is $2 + 1 = 3$.

Now let $m \geq 1$ and assume that  the valence of $z$ in $\G_{m-1}$ is 
less than or equal to 
$2 + m-1$. 

After rechoosing our representatives and reordering the vertices $b_1, \ldots, b_{m-1}$, as well as the $b_m, \ldots, b_n$, we may assume that there is a sequence $\{v_i\}_{i =1}^{\ell}$ for $b_{m-1}$ and $b_m$ where $b_{m-1}$ and $b_m$ are the only branch vertices appearing as coordinates in the sequence and each appears only \textit{once}.  
To see this, take $\{v_ i\}$ to be a sequence shortest among all sequences between vertices $a$ and $b$ such that $\st(a)$ is identified with the star of one of $b_1, \ldots, b_{m-1}$ and $\st(b)$ is identified with the star of one of $b_m, \ldots, b_{n}$.

Now, all of the $v_i$ are $2$--valent and of the same type.
It is now easy to see that $z$ is at most $4$--valent in the image of $\st(b_{m-1}) \cup \st(b_m)$ in $\G_m$---again we are using the fact that each edge of $\GH \sqcup \GK$ embeds in $\T$.
This means that $z$ is at most $(2 + m)$--valent in $\G_m$, and we are done by induction.
\end{proof}

\begin{proof}[Proof of Theorem \ref{eulertheorem}] 
Suppose that there is a $2$--valent vertex $a$ in $\GH \sqcup \GK$ carried to a branch vertex in $\T$ whose star is not carried into the star of any branch vertex.

Let $\{v_i\}_{i=1}^n$ be a sequence for $a$ and a branch vertex $x$ that is minimal among all sequences for $a$ and branch vertices---such a sequence exists by Corollary \ref{nospecials}.
So $x$ is the only branch vertex that appears as a coordinate in the sequence and it only appears once, in $v_n$.

Let $A$ be the type of $a$. 
If there were a $2$--valent vertex of type $B \neq A$ appearing as a coordinate in the sequence, then there would be a term in the sequence with a coordinate of type $A$ and a $2$--valent coordinate of type $X \neq A$, making this term in the sequence extremal, which is impossible.
So every $2$--valent coordinate in the sequence is of type $A$.
It follows that the stars of all of the $2$--valent coordinates in the sequence are identified in $\T$.

But $v_n$ is a $2$--valent vertex of $\GHK$, as only one of its coordinates is a branch vertex. So the star of the $2$--valent coordinate of $v_n$ is carried into the image of the star of $x$.
We conclude that the star of $a$ is carried into the image of the star of $x$, a contradiction.

It follows from this and Corollary \ref{nospecials} that for each branch vertex $z$ in $\T$, we have
\[
\val_\T(z)  = \val_{\G^z}(z).
\]
So, by Lemma \ref{staridentification}, we have
\[
\val_\T(z) \leq 2 + \#\big\{ [\st(b)]_\mathfrak{R}\ \big | \ b \in \GH \sqcup \GK  \mathrm{,}\ \mathrm{valence}(b)=3   \mathrm{,\ and}\ b \mapsto z
\big\}.
\]
We conclude that
\begin{align*}
-\chi(\T) & = \frac{1}{2} \sum_{z\ \mathrm{vertex}} (\val_\T(z) - 2) \\
& \leq \frac{1}{2} \# \big\{ [\st(b)]_\mathfrak{R}\ \big | \ b \in \GH \sqcup \GK \  \mathrm{and}\ \mathrm{valence}(b)=3 \big\}. \qedhere
\end{align*}
\end{proof}

\subsection{Matrices}

Let $X =\{x_1, \ldots, x_{2h-2}\}$ and $Y=\{y_1, \ldots, y_{2k-2}\}$ be the sets of branch vertices of $\GH$ and $\GK$, respectively.
Define a function $f \co X \times Y \to \{0,1\}$ by declaring $f(x_i,y_j) =1$ if $(x_i,y_j)$ is a branch vertex of $\GHK$, zero if not.

Consider the $(2h-2) \times (2k-2)$--matrix $M = \big(f(x_i,y_j)\big)$.
Note that $\sum_{i,j} f(x_i,y_j) = \mathbf{b}(\GHK)$.
In particular, H. Neumann's inequality \cite{hanna,hannaaddendum}
\[
\rank(H\cap K) - 1 \leq 2(h - 1)(k - 1)
\]
becomes the simple statement that the entry--sum of $M$ is no more than $(2h-2)(2k-2)$.

\begin{lemma}\label{normalformlemma}
After permuting its rows and columns, we may assume that $M$ is in the block form
\begin{equation}\label{normalform}
\left( \begin{array}{cccccc}
M_1 &    &  &\\
&  \ddots & &\\
&& M_\ell &\\
&&  & \mbox{\large $0$}_{p\times q}
\end{array} \right)
\end{equation}
where every row and every column of every $M_i$ has a nonzero entry and
\[
\ell + p + q = \# \big\{ [\st(b)]_\mathfrak{R}\ \big | \ b \in \GH \sqcup \GK \ \ \mathrm{and}\ \mathrm{valence}(b)=3 \big\}. 
\] 
\end{lemma}
\noindent When $p$ or $q$ is zero, the notation means that $M$ possesses $q$ zero--columns at the right or $p$ zero--rows at the bottom, respectively. 
\begin{proof}  
Let 
\[
\{e_1,\ldots, e_s\} = \big\{ [\st(b)]_\mathfrak{R}\ \big | \ b \in \GH \sqcup \GK \ \ \mathrm{and}\ \mathrm{valence}(b)=3 \big\},
\]
let $\{r_{i,j}\}_{j=1}^{m_i}$ be the set of rows corresponding to branch vertices in $\GH$ of class $e_i$, and let $\{c_{i,t}\}_{t=1}^{n_i}$ be the set of columns corresponding to branch vertices in $\GK$ of class $e_i$. 

By permuting the rows we may assume that the $r_{1,j}$ are the first $m_1$ rows, the $r_{2,j}$ the next $m_2$ rows, and so on.
Now, by permuting columns, we may assume that the $c_{1,k}$ are the first $n_1$ columns, the $r_{2,k}$ the next $n_2$ columns, and so forth.
Moving all of the zero--rows to the bottom, and all of the zero--columns to the right, we obtain our normal form \eqref{normalform}.

To see that the stated equality holds, first notice that the normal form and the definition of the equivalence relation $\mathfrak{R}$ together imply that: there are precisely $p$ branch vertices in $\GH$ whose stars are not $\mathfrak{R}$--equivalent to that of \textit{any} other branch vertex in $\GH \sqcup \GK$, corresponding to the the $p$ zero--rows at the bottom; and there are precisely $q$ branch vertices in $\GK$ whose stars are not $\mathfrak{R}$--equivalent to that of \textit{any} other branch vertex in $\GH \sqcup \GK$, corresponding to the the $q$ zero--columns at the right.
After reordering the $\mathfrak{R}$--equivalence classes, we may thus list them as 
\[
e_1, \ldots, e_L; \ e_{L + 1}, \ldots, e_{L + p}; \ e_{L + p + 1}, \ldots, e_{L + p + q}
\]
where $e_{L + 1}, \ldots, e_{L + p}$ are the classes of the branches corresponding to the last $p$ rows, and $e_{L + p + 1}, \ldots, e_{L + p + q}$ are the classes corresponding to the last $q$ columns.

By construction of the normal form $M$, each block represents an $\mathfrak{R}$--equivalence class of stars: if the entries $(a,b)$ and $(c,d)$ of $M$ lie in a block $M_i$, then the vertices $x_a$, $y_b$, $x_c$, and $y_d$ all represent the same $\mathfrak{R}$--equivalence class. 
Furthermore, distinct blocks represent distinct classes.
So the number $L$ is at least $\ell$.

Finally, as an equivalence class either corresponds to a block (the equivalence class has representatives in $\GH$ \textit{and} $\GK$), a zero--row (the equivalence class has a unique representative in $\GH$), or a zero--column (the equivalence class has a unique representative in $\GK$), we conclude that $L = \ell$.  
\end{proof}

We will make repeated use of the following lemma.

\begin{lemma}\label{entrysum}
The entry--sum of $M$ is less than or equal to the entry--sum of the $(2h-2) \times (2k-2)$--matrix
\[
\left( \begin{array}{ccc}
\mbox{\large $1$}_{m\times n} &    &  \\
& \diagmatrix & \\
 & &\! \! \! \! \mbox{\large $0$}_{p\times q}
\end{array} \right)
\]
where $m = 2h-2 - p - (\ell -1)$, $n = 2k-2 - q - (\ell-1)$, and $\mbox{\large $1$}_{m\times n}$ is the $m \times n$--matrix all of whose entries are $1$.
\end{lemma}
\begin{proof} We perform a sequence of operations to $M$ that do not decrease the entry--sum and which result in the matrix displayed in the lemma.

First replace each block in $M$ with a block of the same dimensions and whose entries are all $1$.
Of course, this does not decrease the entry--sum.

Now, reorder the blocks in order of nonincreasing entry--sum.
If there are only $1 \times 1$--blocks, we are done.
If all but one of the blocks are $1 \times 1$, we are again done.
So we may assume that at least two blocks have more than one entry.

Let $M_t$ be the last block with more than one entry.
Say that $M_{t -1}$ and $M_t$ are $a \times b$ and $c \times d$ matrices, respectively. 
We now replace $M_{t - 1}$ with an $(a + c - 1) \times (b + d -1)$--block all of whose entries are $1$, and replace $M_t$ with a $1 \times 1$--block whose entry is $1$.
That this does not decrease the entry--sum is best understood using a diagram, which we have provided in Figure \ref{countingfigure}.
\begin{figure} 
\begin{center}
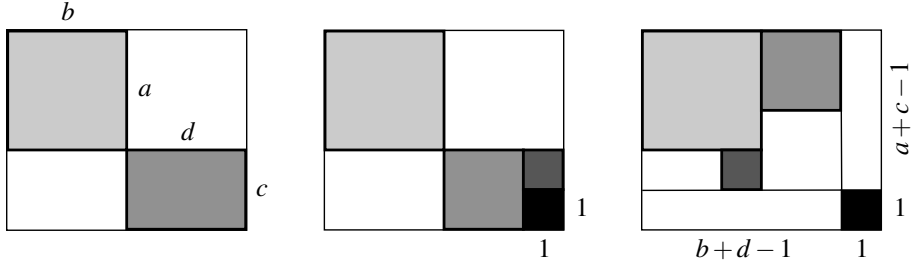
\end{center}
\caption{At left are the blocks $M_{t -1}$ and $M_t$.
The shaded regions represent nonzero entries.  
By the ordering of the blocks, $cd \leq ab$, and so we may assume without loss of generality that $c \leq a$.  
At right we have cut and rearranged these regions to demonstrate that the entry--sum is not decreased by our move.}
\label{countingfigure}
\end{figure}

Repeating this procedure eventually terminates in the matrix displayed in the lemma, and the proof is complete.
\end{proof}

\section{Burns' theorem}\label{burnssection}

We record here a proof of Burns' theorem that requires only the matrix $M$ and a simple count---we recommend B. Servatius' \cite{servatius} and P. Nickolas' \cite{nickolas} proof of this theorem, which involves a clever consideration of a minimal counterexample, as do we recommend the discussion of said argument in \cite{walter}.
To our knowledge, the argument here is new.

\begin{lemma}\label{noPandQ} If $p = q = 0$, then $\ell > 1$.
\end{lemma}
\begin{proof} If $p=q=0$ and $\ell =1$, then $\T$ has a single branch vertex of valence $3$, which is impossible, as it has no extremal vertices.
\end{proof}

\begin{theorem}[Burns] Let $H$ and $K$ be nontrivial finitely generated subgroups of a free group with ranks $h$ and $k$, respectively.
Then
\[
\rank(H\cap K) - 1 \leq 2(h - 1)(k - 1) - \min\big\{(h - 1), (k - 1)\big\}.
\]
\end{theorem}
\begin{proof} Let $h$ and $k$ be the ranks of $H$ and $K$ with $h \leq k$.

If one of $p$ or $q$ is nonzero, then $M$ has a zero--row or a zero--column, by \eqref{normalform}. 
Since $M$ is a $(2h-2) \times (2k-2)$--matrix with entries in $\{0,1\}$ and entry--sum $\mathbf{b}(\GHK) = -2\chi(\GHK)$, we are done. 

So, by Lemma \ref{noPandQ}, we may assume that $\ell \geq 2$, and the  comparison of entry--sums in Lemma \ref{entrysum} yields
{\small \begin{align*}
\mathbf{b}(\GHK) & \leq  \ell -1 + \big(2h-2 - (\ell-1)\big)\big(2k-2 - (\ell-1)\big)\\ 
& = \ell - 1 + (2h-2)(2k-2) - (\ell-1)(2h-2) - (\ell-1)(2k-2) + (\ell-1)^2 \\
&\leq 4(h - 1)(k - 1)  - (2h-2) + \big(\ell - (2k-2)\big)(\ell-1)\\
& \leq 4(h-1)(k-1) - (2h-2),
\end{align*}
}as desired---the inequality is again more easily understood using a diagram, which we have provided in Figure \ref{BurnsCount}.
\end{proof}
\begin{figure} 
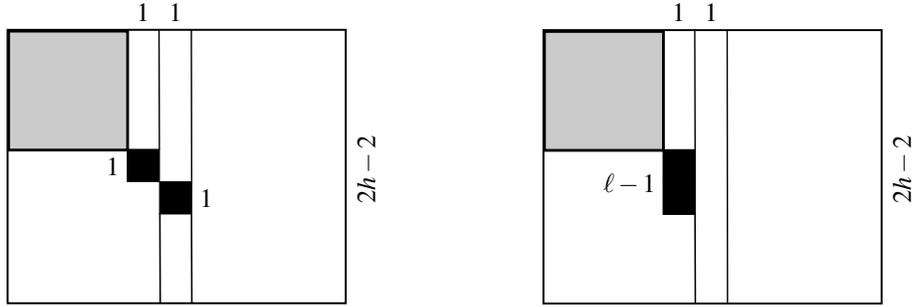
\caption{At left is a graphical representation of the matrix in Lemma \ref{entrysum}.  The shaded regions represent nonzero entries. When $\ell \geq 2$, we may rearrange these regions as shown at the right, establishing the desired inequality.}
\label{BurnsCount}
\end{figure}

\noindent Notice that nothing prevents us from considering pushouts along disconnected graphs, and so we in fact obtain W. Neumann's \cite{walter} strong form of Burns' inequality:
\[
\sum_{\overset{g\in H \backslash F \slash K}{H \cap K^g \neq 1}} \! \! \big( \rank(H\cap K^g) - 1\big) 
\leq 2(h - 1)(k - 1) - \min\big\{(h - 1), (k - 1)\big\}.
\]

\section{Strengthening Burns' inequality}

\begin{theorem}\label{joinburns} Let $H$ and $K$ be nontrivial finitely generated subgroups of $F$ of ranks $h$ and $k \geq h$ that intersect nontrivially.  Then
\[
\rank(H\cap K) -1 \leq 2(h -1)(k-1) - (h -1)\big(\rank(H\vee K) - 1\big).
\]
\end{theorem}
\begin{proof} 
First suppose that $\rank(H\cap K) = 1$.
The desired inequality is then
\[
0 \leq (h -1)\big(2k - \rank(H\vee K) - 1\big).
\]
If 
\[
2k - \rank(H\vee K) - 1 \geq 0,
\]
then we are done.
If this is not the case, then we must have $\rank(H\vee K) = 2k$, and hence $h=k$, since $\rank(H\vee K) \leq h + k$.
Since finitely generated free groups are Hopfian, we must conclude that $\rank(H\cap K) = 0$, which contradicts our assumption.

So we assume as we may that $\rank(H\cap K) \geq 2$. 
As every branch vertex of $\GHK$ is associated to a block of our normal form $M$, this implies that $\ell \geq 1$.

First note that $2h-2 > p + \ell -1$ and $2k-2 > q + \ell -1$.
By Lemma \ref{entrysum}, we have
\begin{align}
\mathbf{b}(\GHK)  & \leq  \ell - 1 + \big(2h-2 - (p + \ell-1)\big)\big(2k-2 - (q+\ell-1)\big) \notag \\
& = \ell -1 + (2h-2)(2k-2) - (p + \ell-1)(2k-2) \notag \\
& \quad \quad \quad \ + \big[ (p + \ell-1)(q + \ell-1) - (2h-2)(q + \ell-1) \big] \notag \\
& \leq \ell -1 + (2h-2)(2k-2) - (p + \ell-1)(2k-2) \notag \\
& \quad \quad \quad \  - [\ell -1]  \notag \\
& = (2h-2)(2k-2) - (p + \ell-1)(2k-2).\label{estimate1}
\end{align}
The proof of the inequality is illustrated in Figure \ref{StrongCount}.
\begin{figure} 
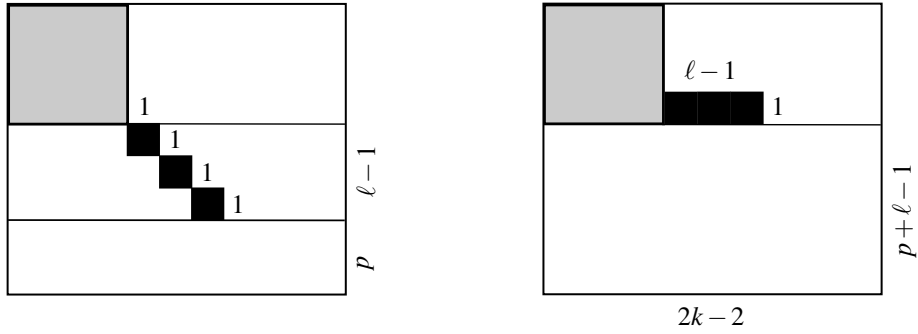
\smallskip
\caption{Again we have the matrix of Lemma \ref{entrysum} at the left, shaded regions representing nonzero entries. Provided $\ell \geq 1$, we may rearrange these regions as shown at the right, establishing \eqref{estimate1}.}
\label{StrongCount}
\end{figure}
Similarly,
\begin{equation}\label{estimate2}
\mathbf{b}(\GHK) \leq (2h-2)(2k-2) - (q + \ell-1)(2h-2).
\end{equation}

Since $\ell \geq 1$, the inequality \eqref{estimate2} provides the theorem unless 
\[
q < \rank(H\vee K) - 1 \leq -\chi(\T).
\]
So we assume that $q \leq -\chi(\T) - 1$, the rest of the argument proceeding as in \cite{imrichmuller}.

By Theorem \ref{eulertheorem} and Lemma \ref{normalformlemma}, we also have $\ell + p + q \geq -2\chi(\T)$ and so 
\[
\ell + p \geq -\chi(\T) + 1 \geq \rank(H\vee K).
\]
By \eqref{estimate1}, we now have 
\[
\mathbf{b}(\GHK) \leq (2h-2)(2k-2) - \big(\rank(H\vee K) -1\big)(2k-2),
\]
and since $k \geq h$, the proof is complete.
\end{proof}

\section{A particular case of the Hanna Neumann Conjecture}\label{particular}

\begin{theorem}[Kent, Louder]\label{bigjoin} Let $H$ and $K$ be nontrivial finitely generated subgroups of a free group of ranks $h$ and $k$, respectively.
If 
\[
\rank(H\vee K) - 1 \geq \frac{h+k-1}{2}
\]
then
\[
\rank(H\cap K) - 1 \leq (h-1)(k-1).
\]
\end{theorem}
\begin{proof} 
Note that if $q \geq k$, then the $(2h-2)\times(2k-2)$--matrix $M$ has at least $k$ zero--columns. 
As $\mathbf{b}(\GHK)$ is the entry--sum of $M$, we have
\[
\mathbf{b}(\GHK) \leq (2k-2 - k)(2h-2) = (k-2)(2h-2),
\]
and so
\[
\rank(H\cap K) - 1 \leq (h-1)(k-2),
\]
which is better than desired.

So assume that $q \leq k-1$.
Then, by assumption, Theorem \ref{eulertheorem}, Lemma \ref{normalformlemma}, and the fact that $-\chi(\T) \geq  \rank(H\vee K) - 1$, we have
\begin{equation}\label{hypoineq}
\ell -1 + p + q \geq h + k - 2.
\end{equation}
Note that $2h-2 > p + \ell -1$.
So, by Lemma \ref{entrysum} and \eqref{hypoineq}, we have
\begin{align*}
\mathbf{b}(\GHK)  & \leq  \ell - 1 + \big(2h-2 - (p + \ell-1)\big)\big(2k-2 - (q+\ell-1)\big) \\
& = \ell -1 + (2h-2)(2k-2) - (p + \ell-1)(2k-2)\\
& \quad \quad \quad \ + \big[ (p + \ell-1)(q + \ell-1) - (2h-2)(q + \ell-1) \big] \\
& \leq \ell -1 + (2h-2)(2k-2) - (h+k-2-q)(2k-2)\\
& \quad \quad \quad \  - [\ell -1] \\
& \leq (2h-2)(2k-2) - (h-1)(2k-2)\\
& = 2(h-1)(k-1). \qedhere
\end{align*} 
\end{proof}

\noindent We do not obtain the stronger inequality 
\[
\sum_{\overset{g\in H \backslash F \slash K}{H \cap K^g \neq 1}} \! \! \big( \rank(H\cap K^g) - 1\big) 
\leq (h - 1)(k - 1)
\]
here, nor the analogous inequality in Theorem \ref{joinburns}, as the pushout along a larger graph could have Euler characteristic dramatically smaller in absolute value than $-\chi(\GHvK)$.
For example, it is easy to find $H$ and $K$ and $u$ and $v$ such that the pushouts $\T^{\, uv}$ and $\T$ of $\GH$ and $\GK$ along $\Gamma_{\! H^u \cap K^v}$ and $\GHK$, respectively, satisfy 
$-\chi(\T^{\, uv}) \geq -\chi(\Gamma_{\! H^u \vee K^v}) \gg  -\chi(\T)$.
As a consequence, the pushout along $\GHK \sqcup \Gamma_{\! H^u \cap K^v}$ will have Euler characteristic much smaller in absolute value than $-\chi(\Gamma_{\! H^u \vee K^v})$.

If the reader would like a particular example of this phenomenon, she may produce one as follows.  

Begin with subgroups $A$ and $B$ of large rank so that the topological pushout of $\Gamma_{\! A}$ and $\Gamma_{\! B}$ is the wedge of two circles: take $A$ and $B$ to be of finite index in $F$, the subgroup $A$ containing $a$, the subgroup $B$ containing $b$.

Now consider the endomorphism $F \to F$ that takes $a$ and $b$ to their squares.
Let $H$ and $K$ be the images of $A$ and $B$ under this endomorphism, respectively.
It is a simple exercise to see that the pushout $\T$ of $\GH$ and $\GK$ along $\GHK$ is homeomorphic to that of $\Gamma_{\! A}$ and $\Gamma_{\! B}$ along $\Gamma_{\! A\cap B}$---it is a wedge of two circles labeled $a^2$ and $b^2$.
It is also easy to see that the pullback $\g$ contains an isolated vertex $(x,y)$, where $x$ is the $2$--valent center of a segment labeled $a^2$ and $y$ is the $2$--valent center of a segment labeled $b^2$.

We may conjugate $H$ and $K$ by elements $u$ and $v$ in $F$, respectively, so that $\Gamma_{\! H^u}= \GH$, $\Gamma_{\! K^v} = \GK$, and $\Gamma_{\! H^u \cap K^v}$ is our isolated point.
So the pushout $\T^{\, uv}$ of $\Gamma_{\! H^u}$ and $\Gamma_{\! K^v}$ along $\Gamma_{\! H^u \cap K^v}$ is the wedge of $\Gamma_{\! H^u}$ and $\Gamma_{\! K^v}$.
In fact, by our choice of isolated point, the pushout $\T^{\, uv}$ will be equal to $\Gamma_{\! H^u \vee K^v}$, as the former admits no folds.

In such a case, the pushout of $\Gamma_{\! H^u}$ and $\Gamma_{\! K^v}$ along $\Gamma_{\! H^u \cap K^v} \sqcup \GHK$ has Euler characteristic small in absolute value (being a quotient of $\T$, it has no more than four edges), despite the fact that the graph $\Gamma_{\! H^u \vee K^v}$ has Euler characteristic very large in absolute value.

See Section \ref{walterineq} for what can be said about the general situation.

\section{Remarks.}

\subsection{A bipartite graph}

Estimating $-\chi(\T)$ may be done from a different point of view, suggested to us by W. Dicks---compare \cite{dicks}.

Given our subgroups $H$ and $K$, define a bipartite graph $\Delta$ with $2h-2$ black vertices $x_1, \ldots, x_{2h-2}$ and $2k-2$ white vertices $y_1, \ldots, y_{2k-2}$ where $x_i$ is joined to $y_j$ by an edge if and only if the $i,j$--entry of $M$ is $1$. 
It is easy to see that the number $c$ of components of $\Delta$ is equal to $\ell + p + q$, and that its edges are $2 \, \rank(H\cap K)-2$ in number.

One may estimate the number of edges of $\Delta$, and hence $\rank(H\cap K)$, by counting the maximum number of edges possible in a bipartite graph with $2h-2$ black vertices and $2k-2$ white vertices whose number of components is equal to $c$.

It may be that a direct study of $\Delta$ would produce the inequalities given here, but we have not investigated this.

\subsection{Walter Neumann inequalities}\label{walterineq}

Let $X$ be a set of representatives for the double coset space $H \backslash F \slash K$ and let $Y$ be the subset of $X$ consisting of those $g$ such that $H \cap K^g$ is nontrivial.
As mentioned at the end of Section \ref{particular}, other than in our treatment of Burns' theorem, we have not estimated the sum
\[
\sum_{g\in Y} \! \! \big( \rank(H\cap K^g) - 1\big) 
\]
using hypotheses on $\rank(H\vee K)$. 
However, we are free to replace $\rank(H\cap K) -1$ with this sum throughout provided we replace $\rank(H\vee K)$ with $\rank \langle H, K, Y \rangle$.

To see this, note that we may replace $\T$ with the pushout $\mathcal S$ of the diagram
 \[
 \xymatrix{\displaystyle{\bigsqcup}\ \Gamma_{H\cap K^g}  \ar[r]  \ar[d]_<{\overset{}{g \in Y}\ \ }
 &   \GH  \\
 \GK
    } 
\]
to obtain a diagram
\[
 \xymatrix{\displaystyle{\bigsqcup}\ \Gamma_{H\cap K^g}  \ar[r]  \ar[d]_<{\overset{}{g \in Y}\ \ } &   \GH   \ar[d] \ar[ddr] \\
  \GK \ar[r] \ar[drr] & \mathcal S \ar[dr] \\
  & & \Gamma_{\! \langle H, K, Y \rangle}
    } 
\]
where the map $\mathcal S \to  \Gamma_{\! \langle H, K, Y \rangle}$ factors into a series of folds.

\bibliographystyle{plain}
\bibliography{join}

\bigskip

\noindent Department of Mathematics, Brown University, Providence, RI 02912
\newline \noindent  \texttt{rkent@math.brown.edu}

\end{document}